\documentclass[11pt]{amsart}

\usepackage{macros}

\usepackage[alphabetic,initials]{amsrefs}

\makeatletter
\let\oldabs\abs
\def\abs{\@ifstar{\oldabs}{\oldabs*}}
\makeatother

\DeclareMathOperator{\ch}{\mathrm{ch}}

\title{Nowhere vanishing 1-forms on 4-folds}
\author{Benjamin Church}
\address{Department of Mathematics, Stanford University, 
380 Jane Stanford Way, CA 94086} 
\email{{\tt bvchurch@stanford.edu}}

\date{\today}

\subjclass[2020]{Primary: 14J10. Secondary: 14D06, 14K05.}

\begin{document}

\begin{abstract}
In this note, we prove -- in dimension at most $4$ -- a conjectue of Hao which says that a morphism $f : X \to A$ to a simple abelian variety $A$ is smooth if and only if there is a $1$-form pulled back from $A$ without any zeros. We also give a complete classification of $4$-folds with a $1$-form without zeros not admitting a map to an elliptic curve.
\end{abstract}

\maketitle

\makeatletter
\newcommand\@dotsep{4.5}
\def\@tocline#1#2#3#4#5#6#7{\relax
  \ifnum #1>\c@tocdepth 
  \else
    \par \addpenalty\@secpenalty\addvspace{#2}%
    \begingroup \hyphenpenalty\@M
    \@ifempty{#4}{%
      \@tempdima\csname r@tocindent\number#1\endcsname\relax
    }{%
      \@tempdima#4\relax
    }%
    \parindent\z@ \leftskip#3\relax
    \advance\leftskip\@tempdima\relax
    \rightskip\@pnumwidth plus1em \parfillskip-\@pnumwidth
    #5\leavevmode\hskip-\@tempdima #6\relax
    \leaders\hbox{$\m@th
      \mkern \@dotsep mu\hbox{.}\mkern \@dotsep mu$}\hfill
    \hbox to\@pnumwidth{\@tocpagenum{#7}}\par
    \nobreak
    \endgroup
  \fi}
\def\l@section{\@tocline{1}{0pt}{1pc}{}{\bfseries}}
\def\l@subsection{\@tocline{2}{0pt}{25pt}{5pc}{}}
\makeatother


\tableofcontents

\section{Introduction}

Advances in Hodge theory, culminating in the celebrated work of Popa and Schnell \cite{PS14}, showed that existence of $1$-forms without zeros on a smooth projective variety bounds its Kodiara dimension. In \cite{Church24, CCH23} and independently \cite{HWZ24} the authors show that such $1$-forms drastically constrain the geometry of the underlying variety. We take the guiding principle that $1$-forms without zeros should be explained geometrically by the existence of a smooth map to an abelian variety. In the non-minimal case, this is not true on the nose (see \cite[Example 1.1]{Church24}) and one must first pass to a birational model in order to obtain a smooth map to an abelian variety. In \cite{Church24}, the author proved that, assuming the existence of good minimal models, any $X$ admitting a $1$-form without zeros is birationally equivalent to an isotrivial smooth fibration over an abelian variety. More generally, Hao and Schreieder \cite{HS21(1)} made the following conjecture.

\begin{Lconj}\cite[Conjectures 1.7 and 1.8]{HS21(1)} \label{conj:HS21}
Let $X$ be a smooth projective variety that admits a holomorphic $1$-form without zeros. Then $X$ is birational to a smooth projective variety $X'$ with a smooth morphism $X' \to A$ to a positive-dimensional abelian variety $A$. Moreover, if $\kappa(X) \ge 0$ we can choose $X' \to A$ to be isotrivial.
\end{Lconj}

Note that birational isotriviality of $X \to A$ is already expected from the results and conjectures of \cite{MP21, Taji23}. In fact, \cite{MP21} predict a stronger conclusion: that $X$ becomes birational to a product after pulling back along an \textit{isogeny} $A' \to A$ rather than an arbitrary \etale cover. When $X$ admits a good minimal model, the author verified this stronger conclusion in \cite{Church24}. Here, we attempt to explicitly understand the birational transformations $X \birat X'$ needed to arrive at a model $X' \to A$ equipped with an isotrivial (not just birationally isotrivial) fibration over an abelian variety. In the case of surfaces and $3$-folds, \cite{HS21(1)} proves that the birational modification $X \birat X'$ in Conjecture~\ref{conj:HS21} can be taken to be a sequence of blow-downs whose centers are embedded elliptic curves mapping nontrivially to $\Alb_X$. Here we consider the case of $4$-folds.
 \par 
 The results of \cite{Church24} hold unconditionally for $4$-folds using the results of Fujino \cite{Fuj10} on abundance for irregular $4$-folds. In particular, the ``$\kappa(X) \ge 0$'' case of Conjecture~\ref{conj:HS21} holds for $4$-folds by \cite[Theorem 4.1]{Church24}. Except possibly when $X$ admits a map to an elliptic curve, we verify the corresponding remaining cases of Conjecture~\ref{conj:HS21} for $4$-folds and give an explicit blow-down description of the modification $X \birat X'$.

\begin{Lthm}[=Theorem~\ref{thm:structure_no_elliptic_factor}]\label{into:thm:structure_no_elliptic_factor}
Let $X$ be a smooth projective $4$-fold with a nowhere vanishing holomorphic $1$-form $\omega$. Assume $\Alb_X$ does not contain an elliptic curve factor. Then there exists a sequence of blow-downs $\pi : X \to X'$ whose centers are smooth surfaces along which $\omega$ is nonvanishing, and a smooth map $X' \to A$ to an abelian variety.
\end{Lthm}

Hao made a related conjecture giving more precise control over the structure in Conjecture~\ref{conj:HS21} when the abelian variety is simple:

\begin{Lconj} \cite[Conj.~1.5]{Hao23}.
Let $f : X \to A$ be a morphism from a smooth projective variety $X$ to a simple abelian variety $A$. The following are equivalent:
\begin{enumerate}
\item there exists a holomorphic $1$-form $\omega \in H^0(A, \Omega^1_A)$ such that $f^{\ast}\omega$ is nowhere vanishing;
\item $f : X \to A$ is smooth.
\end{enumerate}
\end{Lconj}

We prove this conjecture for $4$-folds and give a complete classification of the possible examples.

\begin{Lthm}[=Theorem~\ref{thm:smooth_map_to_simpleAV}]\label{intro:thm:smooth_map_to_simpleAV}
Let $f : X \to A$ be a morphism from a smooth projective 4-fold $X$ to a simple abelian variety $A$. The following are equivalent:
\begin{enumerate}
\item[(a)] there exists a holomorphic one-form $\omega \in H^0(A, \Omega^1_A)$ such that $f^{\ast}\omega$ is nowhere vanishing;
\item[(b)] $f : X \to A$ is smooth.
\end{enumerate}
Moreover, when this holds then $X$ has the following structure.
\begin{enumerate}
    \item If $\dim{A} = 2$ then there exists a factorization $X \to X^{\min} \to A$ such that $X \to X^{\min}$ is a composition of blowups at \etale multi-sections of $X^{\min} \to A$, and one of the following holds:
    \begin{enumerate}
        \item $K_{X^{\min}}$ is nef, in which case $X^{\min} \cong Z \times^G A'$ where $A' \to A$ is an isogeny, $G \subset \ker{(A' \to A)}$, and $Z$ is a smooth proper surface with a $G$-action.
        \item $X^{\min}$ admits a conic bundle structure $X^{\min} \to Y$ and the Stein factorization $Y \to A' \to A$ is the composition of a smooth isotrivial fibration in curves $Y \to A'$ and an isogeny $A' \to A$.
        \item The Stein factorization $X^{\min} \to A' \to A$ is the composition of a smooth del Pezzo fibration $X^{\min} \to A'$ and an isogeny $A' \to A$.
    \end{enumerate}
    \item If $\dim{A} = 3$ then the Stein factorization $X \to A' \to A$ is the composition of a smooth isotrivial fibration in curves and an isogeny. 
    \item If $\dim{A} \ge 4$ then $X$ is an abelian variety and $X \to A$ is an isogeny.
\end{enumerate}
\end{Lthm}

The proofs of all these results rely on a careful study of extremal contractions on smooth $4$-folds and the interaction between the exceptional locus and the non-vanishing $1$-form.

\textbf{Conventions.} We work over the complex numbers. A \textit{variety} is an integral separated scheme of finite type over $\CC$. A \textit{minimal model} is a projective variety $X$ with terminal $\Q$-factorial singularities such that $K_{X}$ is nef. We say a minimal model is \textit{good} if $n K_X$ is base-point-free for some $n > 0$.

\textbf{Acknowledgements.} 
I thank Ravi Vakil, Mihnea Popa, and Nathan Chen for many enlightening conversations and comments on this manuscript as well as the Department of Mathematics at Harvard University for its hospitality during the 2024-2025 academic year. During the preparation of this article, the author was partially supported by an NSF Graduate Research Fellowship under grant DMS-2103099 DGE-2146755.

\section{Extremal Contractions and Nonvanishing $1$-Forms}

By a \textit{contraction} we mean a proper morphism $f : X \to Y$ of varieties with $f_* \struct{X} = \struct{Y}$. A contraction is \textit{elementary} if the relative Picard number $\rho(X) - \rho(Y)$ is $1$. We call a contraction \textit{Fano-Mori} (or \textit{good} in the terminology of \cite{AW98}) if $-K_X$ is $f$-ample. An \textit{extremal contraction} is a contraction arising from an extremal $K_X$-negative curve in the cone of effective curves. These are particularly important examples of elementary Fano-Mori contractions. In this section, we will understand the possible types of extremal contractions from a smooth $4$-fold carrying $1$-forms with nonvanishing conditions. First, we review some terminology and fundamental results on the structure of contractions on 4-folds.

\begin{defn}
The \textit{type} of a contraction $f : X \to Y$ is the pair of integers 
\[ (n,m) := (\dim \Exc(f), \dim f(\Exc(f)) \]
where $\Exc(f) \subset X$ is the locus on which $f$ is not an isomorphism. We say a contraction is 
\begin{enumerate}
\item \textit{fiber type} if $\Exc(f) = X$ 
\item \textit{birational} if $f$ is birational i.e.\ $\dim \Exc(f) < \dim{X}$
\item \textit{divisorial} if $\Exc(f)$ is a divisor
\item \textit{small} if $\dim{\Exc(f)} \le \dim{X} - 2$
\end{enumerate}
\end{defn}

Recall that for smooth 3-folds there are no small extremal contractions (hence no flips are necessary until MMP operations introduce singularities) and for smooth 4-folds, small contractions are well behaved according to the following result of Kawamata \cite[Theorem 1.1]{Kaw89}.

\begin{theorem} \label{thm:small_contraction}
Let $X$ be a smooth projective $4$-fold and $f : X \to Y$ a small elementary contraction. Then every connected component $E \subset \Exc(f)$ satisfies $E \cong \P^2$ and $N_{E|X} \cong \struct{\P^2}(-1)^{\oplus 2}$. 
\end{theorem}

After Mori's groundbreaking work understanding contractions and flips on threefolds \cite{Mor82, Mor88}, during the 80s and 90s, a serious effort was made to understand completely the operations of the MMP in dimension $4$. This led to a nearly complete classification of extremal contractions on smooth 4-folds carried out in numerous papers \cite{And85, Kaw89, Kac97, AW98, Tak99}. 
\par
On any non-singular projective variety $X$ of dimension $n$ \cite{And85} showed that type $(n,n-1)$ and type $(n-1,n-2)$ extremal contractions behave analogously to the $3$-fold case \textit{provided} $f : \Exc{f} \to f(\Exc{f})$ is assumed to be equidimensional. Unlike the case of $3$-folds, however, equidimensionality may fail. Explicitly, he proves that equidimensional type $(n,n-1)$ extremal contractions are flat conic bundles over a smooth base (\cite[Theorem 3.1]{And85}) and type $(n-1,n-2)$ contractions are blowups of a smooth base along a smooth subvariety of codimension $2$ (\cite[Theorem 2.3]{And85}). Luckily, the failure of equidimensionality is quite mild in the case $n = 4$. In either type $(4,3)$ or $(3,2)$, the only fibers of dimension $> 1$ are isolated\footnote{Otherwise, there would be a divisor $E \subset X$ contracted to a curve or a point of $Y$. In the divisorial case, this would be an additional irreducible component of $\Exc{f}$ contradicting \cite[Proposition 8-2-1]{mori_program}. In the fiber type case, any curve $C \subset E$ contracted by $f$ is numerically equivalent to a multiple of the general fiber of $f : X \to Y$ so $C \cdot E = 0$. However, if $f : X \to Y$ is a morphism of relative dimension $r$ and $E \subset X$ is a divisor such that $\codim_Y{f(E)} > 1$ then there exists a curve $C \subset E$ such that $C \cdot E < 0$. Indeed, let $S$ be a general slice of $\dim{Y} - 1$ divisors in $|f^* H_Y|$ and $r$ divisors in $H$ and $C = E \cap S$ then $S$ is a normal surface not contracted by $f$ but $C$ is contracted by $f$ so $E \cdot C = (C^2)_S < 0$. This means that an elementary contraction of fiber type cannot have a divisor $E \subset X$ contracted to a subvariety of codimension $>1$} fibers of pure (\cite[Proposition 4.1]{AW98}) dimension $2$. Moreover, \cite{Kac97} and \cite{AW98} have classified the possible 2-dimensional fibers that may appear.

\subsubsection{Notation}

For $r \ge 0$ let $\FF_r = \P_{\P^1}(\struct{\P^1} \oplus \struct{\P^1}(r))$ be the Hirzebruch surface. If $r \ge 1$ then $\FF_r$ has a unique section $s_0$ with negative self-intersection $-r$ corresponding to the quotient $\struct{\P^1} \oplus \struct{\P^1}(r) \onto \struct{\P^1}$ and let $f$ be the class of a fiber. The relative tautological bundle $\struct{\FF_r}(1)$ is isomorphic to $\struct{\FF_r}(s_0 + r f)$. Note that $\struct{\FF_r}(K_{\FF_r}) = \struct{\FF_r}(-2) \ot \pi^* \struct{\P^1}(r - 2) = \struct{\FF_r}(-2 s_0 - (r + 2) f)$. For $r \ge 1$ let $\SS_r$ be the projective cone over the rational normal curve $\P^1 \embed \P^{r}$ of degree $r$. Alternatively, we may describe $\SS_r$ as the image of $\FF_r \to \P^{r+1}$ defined by the basepoint-free complete linear series of $\struct{\FF_r}(s_0 + r f)$ which exactly contracts $s_0$. Note that $\struct{\SS_r}(1)$ arising from the restriction of $\struct{\P^{r+1}}(1)$ thus agrees with its other possible definition, the line bundle pulling back to the tautological bundle $\struct{\FF_r}(1)$. Note that $\SS_r$ is normal but not Gorenstein or even lci for $r > 2$. Furthermore, $K_{\SS_r} = - \frac{r+2}{r} H$ where $H = c_1(\cO_{\SS_r}(1))$ is the hyperplane class and the unique exceptional divisor has discrepancy $a = \frac{2}{r} - 1$. Indeed, $\pi : \FF_r \to \SS_r$ is the minimal resolution -- performed via a single blowup of the cone point which is a cyclic quotient $(1,1)_r$ singularity. Note that when $r = 2$ then $\SS_r$ has canonical singularities (indeed a single $A_1$ singularity) and $\FF_2 \to \SS_2$ is crepant but when $r > 2$ the discrepancy is negative. Furthermore, the embedding dimension of the cone point of $\SS_r$ is $r+1$.
 
\subsubsection{Classification of Isolated $2$-Dimensional Fibers}

Here we record some results of \cite{Kac97} and \cite{AW98} on isolated $2$-dimensional fibers of Fano-Mori contractions. The reader should note that when these authors refer to the ``fiber'' they often mean reduced (or in the terminology of \cite{AW98} ``geometric'') fiber rather than the natural scheme-theoretic structure on the fiber.

\begin{prop} \cite[Proposition 4.2.1]{AW98} \label{prop:isolated_fibers}
Let $f : X \to Y$ be a Fano-Mori contraction of a $\Q$-factorial log terminal $4$-fold and $F$ a fiber of (pure) dimension $2$. Then any (reduced) irreducible component $F' \subset F$ is normal and the pair $(F', \struct{F'}(-K_X))$ must be one of the following
\begin{enumerate}
    \item $(\P^2, \struct{\P^2}(e))$ with $e = 1,2$
    \item $(\FF_r, \struct{\FF_r}(s_0 + k f))$ with $r \ge r + 1$ and $r \ge 0$
    \item $(\SS_r, \struct{\SS_r}(1))$ with $r \ge 2$.
\end{enumerate}
\end{prop}

Using this, we prove the main result of this section constraining the possible types of extremal contractions on smooth fourfolds carrying $1$-forms with nonvanishing properties:

\begin{theorem} \label{thm:contractions}
Let $X$ be a smooth projective $4$-fold with a nowhere-vanishing $1$-form $\omega$. Then any extremal contraction $f : X \to Y$ must be of the following form:
\begin{enumerate}
\item $Y$ is a smooth projective $4$-fold with a nowhere vanishing $1$-form $\omega_Y$ pulling back to $\omega$ and $f : X \to Y$ is the blowup along a smooth surface $S \subset Y$ such that $\omega_Y|_S$ is nowhere vanishing
\item $Y$ is a $\Q$-factorial terminal 4-fold with a $1$-form $\omega_Y$ and $f : X \to Y$ is the blowup of a smooth elliptic curve $E \subset Y$ such that $\omega_Y|_E$ is nowhere vanishing and $f^* \omega_Y = \omega$
\item a (flat) conic bundle over a smooth $3$-fold $Y$ carrying a nowhere vanishing $1$-form
\item a (flat) del Pezzo fibration over a smooth surface $Y$ carrying a nowhere vanishing $1$-form
\item a smooth Fano 3-fold fibration over an elliptic curve $Y \cong E$.
\end{enumerate}
\end{theorem}

The main idea will be to show that isolated 2-dimensional fibers of Fano-Mori contractions will force global 1-forms to have a zero. Indeed the following observation is elementary:

\begin{prop} \label{prop:conormal_vanishing}
Let $X$ be a smooth $4$-fold, $\omega$ a nowhere vanishing $1$-form, and $E \subset X$ be a surface satisfying the following properties:
\begin{enumerate}
\item $E$ is rational and has only rational singularities
\item $\Omega_E$ has no torsion
\item $E \embed X$ is a regular embedding
\end{enumerate}
then $N_{E|X}^\vee$ must have a non-vanishing section.
\end{prop}

\begin{proof}
consider the exact sequence
\begin{center}
\begin{tikzcd}
0 \arrow[r] & N_{E|X}^\vee \arrow[r] & \Omega_X|_{E} \arrow[r] & \Omega_E \arrow[r] & 0
\end{tikzcd}
\end{center}
which is exact since $E \embed X$ is a regular embedding. Since $\Omega_E$ has no torsion, any $\omega \in H^0(E, \Omega_E)$ must be nonzero at the generic point and hence pulls back to a nonzero form under a resolution $\wt{E} \to E$. Since $E$ is rational $H^0(\wt{E}, \Omega_{\wt{E}}) = 0$ hence, by the vanishing of torsion, $H^0(E, \Omega_E) = 0$. Therefore, the global form $\omega \in H^0(X, \Omega_X)$ restricted to $\Omega_X|_E$ arises from a non-vanishing section of $N_{E|X}^\vee$.
\end{proof}

To show that the isolated 2-dimensional fibers will violate the conclusion of the above lemma we need a lemma.

\begin{lemma} \label{lemma:fiber_normal_inequality}
Let $f : X \to Y$ be a morphism of varities and $F \subset X$ a maximal-dimensional subscheme of $X_y$ such that $\dim{F} > \dim{X_{y'}}$ for all nearby fibers $y' \in Y \sm \{ y \}$. Suppose that $F \embed X$ is a regular embedding then 
\[ h^0(N_{F|X}) \le h^1(N_{F|X}) \]
\end{lemma}

The idea is a standard technique in deformation theory and appears as a step in the proof of the main theorem of \cite{AW98}. We reproduce a proof here for the reader's convenience.

\begin{proof}
Since $F \embed X$ is a regular embedding, the Hilbert scheme of $X$ has no local obstructions at $[F]$. Therefore, $T^i = H^{i-1}(F, N_{F|X})$ forms a tangent-obstruction theory \cite[Proposition 6.5.2]{FGAexplained}. Hence
\[ \dim_{[F]} \Hilb_X \ge h^0(N_{F|X}) - h^1(N_{F|X}) \]
so if this number is positive then there exists a pointed affine curve $(C, 0)$ and a closed subscheme $Z \subset X \times C$ flat over $C$ with $Z_0 = F$ so that $Z \to X$ has image strictly containing $F$. Consider the diagram,
\begin{center}
\begin{tikzcd}
Z \arrow[rd] \arrow[r] & X \times C \arrow[d] \arrow[r] & Y
\\
& C  
\end{tikzcd}
\end{center}
since $Z_0 = F \subset X_y$ we see that $Z \to Y$ contracts the fiber of $Z \to C$ over $0$ to a point. Since $Z \to C$ is proper, by the rigidity lemma, every fiber of $Z \to C$ is contracted. Since $Z \to C$ is flat $\dim{Z_t} = \dim{F} > \dim{X_{y'}}$ for any nearby fiber so we must have $Z_t \subset X_y$ but this is impossible because we assumed that $F$ is of maximum dimension in $X_y$ and $Z_t$ moves. Therefore $h^0(N_{F|X}) \le h^1(N_{F|X})$. 
\end{proof}

\begin{cor} \label{cor:conormal_nonvanishing}
In the situation of Lemma~\ref{lemma:fiber_normal_inequality} if moreover $F$ is a normal variety, $\codim_F{X} = 2$, $H^1(F, \struct{F}) = 0$, and $H^1(F, \det N_{F|X}) = 0$ (e.g. if $\det N_{F|X}$ is anti-ample) then $N_{F|X}^\vee$ does not admit a nowhere vanishing section.
\end{cor}

\begin{proof}
Since $F \embed X$ is a regular embedding $N_{F|X}$ is a vector bundle. Suppose $N_{F|X}^\vee$ has a nonvanishing section or equivalently a short exact sequence,
\begin{center}
\begin{tikzcd}
0 \arrow[r] & \L \arrow[r] & N_{F|X} \arrow[r] & \struct{F} \arrow[r] & 0
\end{tikzcd}
\end{center}
where $\L = \det N_{F|X}$ is anti-ample by assumption. Therefore the long exact sequence gives
\begin{center}
\begin{tikzcd}
0 \arrow[r] & H^0(F, \L) \arrow[r] & H^0(F, N_{F|X}) \arrow[r] & H^0(F, \struct{F}) \connectingmap{lld}
\\
& H^1(F, \L) \arrow[r] & H^1(F, N_{F|X}) \arrow[r] & H^1(F, \struct{F})
\end{tikzcd}
\end{center}
By assumption $H^0(F, \struct{F}) = 0$ and $H^1(F, \struct{F}) = 0$. Since $\L$ is anti-ample and $F$ is normal $H^1(F, \L) = 0$ by the main result of \cite{DD89}.
Therefore $h^1(F, N_{F|X}) = 0$ and $h^0(F, N_{F|X}) > 1$ contradicting Lemma~\ref{lemma:fiber_normal_inequality}.
\end{proof}

The following lemma will also be used repeatedly in the subsequent discussion:

\begin{lemma} \label{lemma:blowup_nowhere_vanishing}
Let $Y$ be a variety and $f : X \to Y$ be the blowup of a smooth subscheme $Z \embed Y$ such that $X$ is smooth. Let $\omega$ be a global $1$-form on $Y$. Then $f^* \omega$ is nonvanishing if and only if $\omega|_Z$ and $\omega|_{Y \setminus Z}$ are both nonvanishing.
\end{lemma}

\begin{proof}
We apply the argument of Proposition~\ref{prop:conormal_vanishing} upstairs. Let $E$ be the exceptional divisor of $f$ and consider the exact sequence,
\[ 0 \to N^{\vee}_{E|X} \to \Omega_X |_E \to \Omega_E \to 0 \]
which is exact on the left since $E \embed X$ is regularly embedded. Indeed, $X$ is smooth and $E \embed X$ is a Cartier divisor so $N^{\vee}_{E|X} = \struct{E}(-E)$ is a line bundle hence the map to the vector bundle $\Omega_X|_E$ is injective. The $1$-form $f^* \omega \in H^0(X, \Omega_X)$ restricted to $H^0(E, \Omega_X | E)$ has no zero anywhere on $E$. Furthermore, its image $(f^* \omega)|_E$ in $H^0(E, \Omega_E)$ equals $f^* (\omega|_Z)$ and hence is ``constant'' along the fiber. In particular, if $z \in Z$ is a point at which $\omega|_Z$ is zero then $f^* \omega$ restricted to $H^0(E_z, \Omega_X |_{E_Z})$ must lie in $H^0(E_z, N^{\vee}_{E|X}|_{E_z})$ but $N^{\vee}_{E|X} = \struct{E}(-E)$ is $f$-ample so any section in $H^0(E_z, N^{\vee}_{E|X}|_{E_z})$ of its restriction to a fiber must admit a zero. Hence there exists a point $x \in E_z$ so that $f^* \omega$ is zero at $x$. Therefore, the zero locus of $\omega|_Z$ is empty.
\end{proof}

Let $X$ be a smooth projective $4$-fold containing a surface $F' \subset X$. For each pair $(F', \struct{F'}(-K_X))$ among the list of surfaces in Proposition~\ref{prop:isolated_fibers} if $F' \embed X$ is regularly embedded we can analyze $\det N_{F'|X} = \struct{F'}(-K_X + K_{F'})$ as follows:

\begin{enumerate}
    \item $(\P^2, \struct{\P^2}(e))$ for $e = 1,2$ then $F'$ is regular so $F' \embed X$ is a regular embedding and $\det N_{F'|X} = \struct{\P^2}(e - 3)$ is anti-ample so Corollary~\ref{cor:conormal_nonvanishing} applies.
    \item $(\FF_r, \struct{\FF_r}(s_0 + k f))$ for $k \ge r+1$ and $r \ge 0$ then $F'$ is regular so $F' \embed X$ is a regular embedding and $\det N_{F'|X} = \struct{\FF_r}(-s_0 + (k - r - 2) f)$ is not anti-ample. However, if $\pi : \FF_r \to \P^1$ is the natural projection, it does satisfy $H^1(\FF_r, \struct{\FF_r}(-s_0 + (k - r - 2) f)) = H^1(\P^1, \RR \pi_* \struct{\FF_r}(-s_0 + (k - r - 2) f)) = 0$ since on fibers $\struct{\FF_r}(-s_0 + (k - r - 1) f)$ restricts to $\struct{\P^1}(-1)$ which has no cohomology. Therefore Corollary~\ref{cor:conormal_nonvanishing} applies.
    \item $(\SS_r, \struct{\SS_r}(1))$ with $r \ge 2$ then unfortunately $\SS_r$ is not lci and hence not regularly embedded for $r > 2$. However, for $r = 2$ it is lci and $\det N_{F'|X} = \struct{\SS_r}(-3)$ is anti-ample so Corollary~\ref{cor:conormal_nonvanishing} applies.
\end{enumerate}

Therefore, we see that no component on the list of surfaces in Proposition~\ref{prop:isolated_fibers} can appear except possible $\SS_r$ for $r \ge 3$. We make one further observation before proving the main result: if $f : X \to Y$ is a Fano-Mori contraction then its fibers are covered by rational curves by \cite{MM86} (since all contracted curves satisfy $K_X \cdot C < 0$) and hence are contracted under the Albanese map. Therefore by properness of $f : X \to Y$ there is a diagram:
\begin{center}
\begin{tikzcd}
X \arrow[r, "a"] \arrow[d] & \Alb_X
\\
Y \arrow[ru]
\end{tikzcd}
\end{center}
The upshot is that $\omega = a^* \omega_{\Alb_X}$ pulls back through $Y$ meaning there is a $1$-form (an actual global K\"{a}hler differential not just a reflexive $1$-form) $\omega_Y \in H^0(Y, \Omega_Y)$ so that $\omega = f^* \omega_Y$. This implies that $\omega_Y$ is nowhere vanishing (in the sense of being nonzero in $\Omega_{Y,y} = \m_y / \m_y^2$ at each $y \in Y$) although we are mostly interested in this property when we can show that $Y$ is smooth. Now we are ready to prove the main result.

\begin{proof}[Proof of Theorem~\ref{thm:contractions}]
We perform casework on the type $(n,m)$ of the contraction $f : X \to Y$ utilizing the classification of extremal contractions of 4-folds and analyzing particular surfaces $E \subset X$ that appear in the classification. Note that since $X$ is regular, as long as $E$ is lci the embedding $E \embed X$ is regular. 
\par
By Kawamata's result (Theorem~\ref{thm:small_contraction}) if $f$ is small then $X$ contains a surface $E \cong \P^2$ with $N_{E|X} \cong \struct{\P^2}(-1)^{\oplus 2}$ hence its conormal bundle does not admit a nowhere vanishing $1$-form. Since $E$ is smooth and rational it carries no global $1$-forms. This contradicts the above discussion. 
\par 
We now consider the remaining types of extremal contractions:
\begin{enumerate}
    \item[Type $(4,0)$] implies that $X$ is Fano which is impossible since $X$ carries a $1$-form.
    \item[Type $(4,1)$] the base must be a normal curve with a nowhere vanishing $1$-form so $Y \cong E$ is an elliptic curve and $f : X \to E$ is a fibration in Fano 3-folds. Since the $1$-form pulls back to $\omega$ which has no zeros on $X$ and $X$ is smooth we see that $f$ is smooth.
    \item[Type $(4, 2)$] since the Albanese factors through $X \to Y$ by \cite[Lemma 2.3]{HS21(1)} every fiber of $X \to Y$ must be 2-dimensional otherwise some divisor would get contracted. Therefore, by \cite[Lemma 3.10(iii)]{Cas08} $Y$ is a smooth surface and $f : X \to Y$ is an equidimensional -- hence flat by miracle flatness -- de Pezzo fibration.
    \item[Type $(4,3)$] Ando proves \cite[Theorem 3.1]{And85} that if every fiber of $f : X \to Y$ has dimension $\le 1$ then $Y$ is smooth and $f$ is a flat conic bundle. Furthermore, as discussed above, any fiber of dimension $> 1$ must be an isolated $2$-dimensional fiber. By Proposition~\ref{prop:isolated_fibers} each isolated $2$-dimensional fiber $F$ must contain as a reduced irreducible component $F' \subset F$ so that the pair $(F', \struct{F'}(-K_X))$ appears on the list:
    \begin{enumerate}
        \item $(\P^2, \struct{\P^2}(e))$ with $e = 1,2$
        \item $(\FF_r, \struct{\FF_r}(s_0 + k f))$ with $r \ge r + 1$ and $r \ge 0$
        \item $(\SS_r, \struct{\SS_r}(1))$ with $r \ge 2$.
    \end{enumerate}
    we already saw that Corollary~\ref{cor:conormal_nonvanishing} applies to all except $\SS_r$ with $r \ge 3$ and these surfaces are regular except for $\SS_2$ which has an $A_1$-singularity and hence $\Omega_{F'}$ is torsion-free. Therefore, $E = F'$ satisfies the assumptions of Proposition~\ref{cor:conormal_vanishing} contradicting the conclusion of Corollary~\ref{cor:conormal_nonvanishing}. For the case $E = \SS_r$ note that the embedding dimension of the cone point $p \in \SS_r$ is $r+1$. Since $X$ is a regular $4$-fold, $r \le 3$ and for $r = 3$ the cotangent space of $E$ at $p$ spans $\Omega_{X} \to \kappa(p)$. However, $\omega = f^* \omega_Y$ so $\omega(p)$ is zero in the cotangent space of any closed subscheme of the fiber and thus $\omega(p) \in \Omega_{X} \ot \kappa(p)$ is zero contradicting the assumption that $\omega$ is nowhere vanishing. Hence this rules out $\SS_r$ for $r \ge 3$. We have shown there cannot be a $2$-dimensional fiber and therefore $f : X \to Y$ is everywhere a conic bundle and $Y$ is smooth. 
    \item[Type $(3,0)$] is ruled out by \cite[Lemma 2.3]{HS21(1)} applied to the Albanese $X \to \Alb_X$.
    \item[Type $(3,1)$] by the main theorem of \cite{Tak99} $f : X \to Y$ is the blowup of a regularly embedded smooth curve $C \subset Y$ and $Y$ is a terminal $\Q$-factorial $4$-fold. Now $\omega = f^* \omega_Y$ and I claim $\omega_Y|_C$ is nowhere vanishing. Indeed, because $f : X \to Y$ is a resolution of singularities with the structure of a blowup along the smooth center $C$, Lemma~\ref{lemma:blowup_nowhere_vanishing} applies. Therefore, $\omega_Y|_C$ has no zero and hence $C$ is an elliptic curve. 
    \item[Type $(3,2)$] Ando proves \cite[Theorem 2.3] {And85} that if every fiber of $f : X \to Y$ has dimension $\le 1$ then $Y$ is smooth and $f$ is the blowup of $Y$ along a smooth surface $S \subset Y$. By the same logic as in type $(4,3)$ we rule out the existence of any $2$-dimensional fiber so Ando's result applies. Furthermore, since $\omega = f^* \omega_Y$ is a nowhere vanishing $1$-form we conclude that is $\omega_Y |_S$ is also nowhere vanishing by Lemma~\ref{lemma:blowup_nowhere_vanishing}.
\end{enumerate}
\end{proof}

\section{Smoothness in codimension $1$}

In this section, we use a purity result for the singular locus of a morphism between smooth varieties to show that flat morphisms to a simple abelian variety pulling back a $1$-form without introducing zeros must be smooth away from codimension $1$ on the base. This uses a weak form of purity for maps between smooth varieties.

\begin{lemma} \label{lem:purity}
Let $f : X \to Y$ be a dominant morphism of smooth varieties with relative dimension $r = \dim{X} - \dim{Y}$. Then each irreducible component of the singular locus of $f$ has codmension $\le r + 1$ in $X$.
\end{lemma}

\begin{proof}
This follows immediately from the Jacobian criterion. Using the \etale local structure of smooth varieties we can reduce to a map of affine spaces. Then the singular locus is cut out by the $r \times r$ minors of the Jacobian matrix. A much more general result is proved in the main theorem of \cite{Kal13}.
\end{proof}

\begin{prop} \label{prop:smooth_in_codim_1}
Let $f : X \to A$ be a dominant morphism from a smooth proper variety $X$ to a simple abelian variety $A$.  Suppose that there is a $1$-form $\omega \in H^0(A, \Omega_A^1)$ such that $f^* \omega$ has no zeros. Then $f$ is smooth away from codimension $1$ on the base. 
\end{prop}

\begin{proof}
By purity, Lemma~\ref{lem:purity}, the singular locus $Z \subset X$ has each component of codimension at most $n+1$. By generic smoothness, no component of $Z$ dominates $A$. If $f$ is not smooth away from codimension $1$ then there is some irreducible component $Z' \subset Z$ whose image in $A$ is a divisor $D$. By generic smoothness of $Z' \to D$, there is an open set $U \subset D$ which is smooth and such that the smooth locus of $Z'|_U \to U$ intersects every fiber. This means, for each $y \in U$ there exists $x \in Z'|_U$ so that $\d{f} : T_x Z' \onto T_y D$ is surjective. However, if we consider the diagram,
\begin{center}
    \begin{tikzcd}
    T_x Z' \arrow[r, hook] \arrow[d, "\d{f}", two heads] & T_x X \arrow[d, "\d{f}"]
    \\
    T_y D \arrow[r, hook] & T_y A    
    \end{tikzcd}
\end{center}
By the definition of $Z$, the tangent map $\d{f} : T_x X \to T_y A$ is not surjective. However, $T_y D \embed T_y A$ is codimension $1$ so by the surjectivity of $T_x Z' \to T_y D$ it holds that $\im{\d{f}} = T_y D$. 
\par 
On the other hand, our form $f^* \omega$ must be nonzero at $x \in X$. This means $\omega$ cannot vanish on $T_y D \subset T_y A$ since $\d{f}$ sends every tangent vector into $T_y D$. Furthermore, since $X$ is proper, a $1$-form being nowhere vanishing is an open condition on $f^* H^0(A, \Omega_A) \subset H^0(X, \Omega_X)$. Hence the generic global $1$-form on $A$ satisfies the same property, that $\omega|_U$ has no zeros. This implies by \cite[Proposition 3.3]{DHL21} that $g(Z')$ is fibered by Tori which is impossible since it is a divisor in a simple abelian variety. 
\end{proof}

\begin{cor} \label{cor:finite_implies_etale}
If $f : X \to A$ is a generically finite map from a smooth projective variety to a simple abelian variety and there is a $1$-form $\omega \in H^0(A, \Omega_A)$ so that $f^* \omega$ has no zeros then $f$ is a modification followed by an isogeny. If $\dim{X} \le 4$ then the modification is an isomorphism.
\end{cor}

\begin{proof}
The previous result shows that the image of the branch locus has codimension $\ge 2$ on the base. In particular, $f$ is surjective. Now we apply Stein factorization to $f$ in order to obtain maps $X \to X' \to Y$ where $X' \to Y$ is finite. Since $X'$ is normal and $f' : X' \to Y$ is a finite surjective map onto a smooth variety, by Zariski-Nagata purity it is \etale if and only if it is unramified in codimension $1$. Note that if $f : X \to Y$ is smooth (i.e.\ \etale) at $x \in X$ and $x$ has image $x' \in X$ inside the smooth locus of $X'$ then $X' \to Y$ is \etale at $x'$ by a simple argument on tangent spaces. But $X'$ is regular in codimensional $1$ so the divisorial part of the branch locus $D$ of $X' \to Y$ is contained in the image of the singular locus of $f$. Since $f' : X' \to Y$ is finite, $D$ is not contracted onto $Y$ but is contained in the image of the singular locus of $f$. Since $f$ is smooth away from codimension $1$ on the base, this proves that $D$ is empty. An application of Zariski-Nagata purity shows that $f'$ is \etale.
\par
Furthermore, if $\dim{X} \le 4$ then we may apply Theorem~\ref{thm:contractions} or the analogous results for $3$-folds found in \cite[Proposition 3.1]{HS21(1)} (and the Castellounovo contraction theorem for surfaces) to show that $X \to X'$ is factored as a sequence of blowups at subvarieties on which $\omega$ restricts to a nonvanishing $1$-form. Since $X'$ is simple, there are no such subvarieties by, for example, \cite[Proposition 3.3]{DHL21}.
\end{proof}

\begin{rmk} \label{rmk:flat_case}
Suppose $f : X \to A$ is flat. Each component of the singular locus $Z \subset \mathrm{Sing}(f)$ has codimension $\le \dim{X} - \dim{A} + 1$. Hence if $f|_Z$ is generically finite then its image is a divisor. Therefore, if $f$ is smooth in codimension $1$ then each component $Z$ must have positive fiber dimension. Since $f$ is flat, $X_s \cap \mathrm{Sing}(f) = \mathrm{Sing}(X_s)$ so each component of the singular locus of each fiber is $\ge 1$. In particular, if $f : X \to A$ is a flat map from a $4$-fold to an abelian surface all of whose fibers are normal and reduced then smooth in codimension $1$ implies smooth. 
\end{rmk}

\begin{cor}
With the notation of Proposition~\ref{prop:smooth_in_codim_1}, if $X$ (or equivalently the general fiber $F$ of $X \to A$) admits a good minimal model then there is an isogeny $A' \to A$ such that $X \times_A A' \birat F \times A'$. 
\end{cor}

\begin{proof}
This follows immediately from Proposition~\ref{prop:smooth_in_codim_1} and \cite[Theorem C]{Church24}.
\end{proof}

\section{Nowhere vanishing $1$-forms on Mori Fiber Spaces}

\subsection{Conic Bundles}

First we recall some basic facts about conic bundles.

\begin{defn}
A \textit{conic bundle} $f : X \to S$  is a flat proper morphism of finite presentation such that each fiber $X_s$ admits an embedding $X_s \embed \P^2_{\kappa(s)}$ over $\kappa(s)$ as a (possibly singular) conic. 
\end{defn}

\newcommand{\disc}{\mathrm{disc}}

A conic bundle admits a relative dualizing complex $\omega_{X/Y}^\bullet$ which is a line bundle because the fibers are Gorenstein curves. Furthermore, $\omega_{X_s} \cong \cO_{X_s}(-1)$ from the embedding and hence the embedding of each fiber is canonical and $\omega_{X/Y}^\vee$ is $f$-ample. Then $\E = f_* \omega_{X/Y}^\vee$ is a rank $3$ vector bundle whose formation commutes with arbitrary base change because $H^{>0}(X_s, \omega_{X_s}^\vee) = 0$. Finally, the map $X \to \P_Y(\E)$ is a closed embedding which, on fibers, induces the embedding of the conic curve in $\P^2$. Therefore, a conic bundle is equivalently a relative Gorenstein curve with $\omega_{X/Y}^\vee$ relatively very ample. Let $\P := \P_Y(\E)$ and $X \subset \P$ is cut out as a divisor by a section of some bundle. To determine the bundle, we use adjunction
\[ \omega_{X/Y} \cong \omega_{\P/Y} |_{X} \ot \cN_{X|\P} \]
and note that $\omega_{\P/Y} = (\det \E) \ot \cO_{\P}(-3)$ and $\omega_{X/Y} = \cO_{\P}(-1)|_X$ so we find that 
\[ \cN_{X|\P} = [(\det \E)^\vee \ot \cO_{\P}(2)]|_X \]
so $\cO_{\P}(X) \cong (\det \E)^\vee \ot \cO_{\P}(2)$ and $X$ is cut out by a section 
\[ s \in \Gamma(\P, (\det \E)^\vee \ot \cO_{\P}(2)) \]
We will use these relations to compute Chern numbers of $X$. However, these depend not only on $Y$ but also on $c(\E)$ which we will need to relate to another geometric object: the discriminant locus. The discriminant locus 
\[ \Delta_f = \{ y \in Y \mid X_y \text{ is singular} \} \]
is naturally given a scheme structure via the vanishing of a discriminant function
\[ \disc(f) \in \Gamma(Y, (\det \E)^\vee) \]
Indeed, for any $\L \in \Pic(Y)$ and $\P := \P_Y(\E)$ the projectivization of a rank $r$ vector bundle, there is a nonlinear function taking a section $s \in H^0(\P, \pi^* \L \ot \cO_{\P}(n))$ to its discriminant
\[ \disc(s) \in H^0(Y, \L^{r (n-1)^{(r-1)}} \ot (\det \E)^{n(n-1)}) \]

\begin{prop} \cite[Theorem 1.1]{Tan23} \label{prop:discriminat_locus}
Let $f : X \to Y$ be a conic bundle. Then the discriminant locus $\Delta_f \subset Y$ is a reduced divisor and $\Delta_f = - f_* K_{X/Y}^2$. 
\end{prop}

\begin{lemma}
Let $X$ be a smooth proper 4-fold. Then the following are equivalent
\begin{enumerate}
    \item $\chi(X, \Omega_X^p) = 0$ for all $0 \le p \le 4$
    \item the following Chern numbers vanish
    \begin{align*}
    c_1^4 - 3 c_2^2 & = 0 
    \\
    c_1^2 c_2 & = 0
    \\
    c_1 c_3 & = 0
    \\
    c_4 & = 0 
    \end{align*}
\end{enumerate}
\end{lemma}

\begin{proof}
This is an elementary, if tedious, application of Hirzebruch–Riemann–Roch. We compute
\[ \chi(X, \Omega_X^p) = \int_X \ch(\Omega_X^p) \mathrm{Td}_X \]
for each $p$ to obtain
\begin{align*}
\chi(X, \struct{X}) &= - \tfrac{1}{720} c_1^4 + \tfrac{1}{180} c_1^2 c_2 + \tfrac{1}{240} c_2^2 + \tfrac{1}{720} c_1 c_3 - \tfrac{1}{720} c_4   
\\
\chi(X, \Omega_X^1) &= - \tfrac{1}{180} c_1^4 + \tfrac{1}{45} c_1^2 c_2 + \tfrac{1}{60} c_2^2 - \tfrac{1}{90} c_1 c_3 - \tfrac{1}{180} c_4 
\\
\chi(X, \Omega_X^2) &= - \tfrac{1}{120} c_1^4 + \tfrac{1}{30} c_1^2 c_2 + \tfrac{1}{40} c_2^2 - \tfrac{1}{120} c_1 c_3 + \tfrac{1}{120} c_4 
\\
\chi(X, \Omega_X^3) &= - \tfrac{1}{180} c_1^4 + \tfrac{1}{45} c_1^2 c_2 + \tfrac{1}{60} c_2^2 - \tfrac{1}{90} c_1 c_3 - \tfrac{1}{180} c_4 
\\
\chi(X, \Omega_X^4) &= - \tfrac{1}{720} c_1^4 + \tfrac{1}{180} c_1^2 c_2 + \tfrac{1}{240} c_2^2 + \tfrac{1}{720} c_1 c_3 - \tfrac{1}{720} c_4 
\end{align*}
Row reduction shows these are zero if and only if the above Chern numbers vanish. 
\end{proof}

We will use the following lemmma repeatedly to prove the uniruled cases of the classification results.

\begin{lemma} \label{lemma:conic_bundle_formulas}
Let $f : X \to Y$ be a conic bundle between smooth proper varieties with $\dim{X} = 4$. Suppose that $\chi(\Omega_X^p) = 0$ for all $p$ and that $c_2(Y) = c_3(Y) = 0$. Then
\begin{align*}
c_1(\E)^3 - c_1(\E) c_1(Y)^2 & = 0
\\
c_1(\E)^2 c_1(Y) + c_1(\E) c_1(Y)^2 & = 0
\\
c_1(Y)^3 - c_1(Y) c_2(\E) & = 0
\\
c_3(\E) & = 0 
\end{align*}
In particular, if $c_1(Y) = 0$ as well, then $c_1(\E)^3 = c_3(\E) = 0$. Recall that $c_1(\E) = -[\Delta_f]$.
\end{lemma}

\begin{proof}
We compute intersection theory on $X$ via the embedding $X \embed \P_Y(\E)$ which is a smooth Cartier divisor with $\cO_{\P}(X) = (\det{\E})^\vee \ot \cO_{\P}(2)$. The following exact sequences,
\begin{center}
    \begin{tikzcd}
        0 \arrow[r] & \cO \arrow[r] & \pi^* \E^{\vee}(1) \arrow[r] & \T_{\P/Y} \arrow[r] & 0 
        \\
        0 \arrow[r] & \T_{\P/Y} \arrow[r] & T_{\P} \arrow[r] & \pi^* \T_Y \arrow[r] & 0 
        \\
        0 \arrow[r] & \T_X \arrow[r] & \T_{\P}|_X \arrow[r] & \cN_X \arrow[r] & 0
    \end{tikzcd}
\end{center}
together give
\[ c(\T_X) = f^* c(\T_Y) \cdot c(\pi^* \E^{\vee}(1))|_X \cdot c(\cN_X)^{-1} = [\pi^* c(\T_Y) \cdot c(\pi^* \E(1)) \cdot (1 - c_1(\E) + 2 \xi)^{-1}]|_X \]
Since this is pulled back from $\P$ we can compute the intersection numbers of the previous lemma by computing on $\P$ and intersecting with $[X] = -c_1(\E) + 2 \xi$. Performing this calculation using the assumption that $c_2(Y) = c_3(Y)$ we obtain
\begin{align*}
c_1(X)^4 - 3 c_2(X)^2 &=  c_1(\E)^3 + 4 c_1(\E)^2 c_1(Y) + 3 c_1(\E) c_1(Y)^2 + 8 c_1(Y)^3 - 8 c_1(Y) c_2(\E) 
\\
c_1^2 c_2 &= c_1(\E)^2 c_1(Y) + c_1(\E) c_1(Y)^2 + 2 c_1(Y)^3 - 2 c_1(Y) c_2(\E)
\\
c_1 c_3 &= -c_1(\E)^2 c_1(Y) - c_1(\E) c_1(Y)^2 - 2 c_3(\E)
\\
c_4 &= -c_1(\E)^3 - c_1(\E)^2 c_1(Y) + 4 c_3(\E)
\end{align*}
The previous lemma shows that if $\chi(X, \Omega_X^p) = 0$ for all $p$ then these Chern numbers vanish. Linear elimination then give the requisite relations.
\end{proof}

\subsection{del Pezzo Fibrations} \label{section:delPezzo}

Let $f : X \to Y$ be a flat projective morphism with $f_* \cO_X = \cO_Y$. We call it a \textit{del Pezzo fibration} if the generic fiber is a smooth del Pezzo surface. It is \textit{Fano-Mori} if moreover $X, Y$ are normal $\Q$-Gorenstein and $-K_X$ is $f$-nef. 
\par 
Suppose $f : X \to Y$ is smooth in codimension $1$ and that $X$ and $Y$ are smooth varieties. We will investigate what can happen in this case. Let $U \subset Y$ be the locus over which $f$ is smooth. Since $f : X \to Y$ is smooth in codimension $1$, this open set $U$ has complement of codimension $\ge 2$. Then $L = R^2 f_* \ul{\CC}$ is a local system and $\codim_Y(Y \sm U) \ge 2$ meaning $\pi_1(U) \to \pi_1(Y)$ is an isomorphism. Hence $L$ is determined by a representation of $\pi_1(Y)$. Furthermore, over, $U$ the local system $L$ underlies a polarized VHS of type $(1,1)$ and hence has finite monodromy \cite{Kat72}. But $\pi_1(U) \to \pi_1(Y)$ is an isomorphism so there is a finite \etale cover $Y' \to Y$ such that $X_{Y'} \to Y$ is a del Pezzo fibration with trivial $R^2 f_* \ul{\CC}$.
\par 
Following the argument of \cite[Theorem 8.4]{HS21(1)} we now assume the a relative MMP over $Y'$ can be run and obtains a sequence of divisorial contractions $X_{Y'} \to X'$ such that $X'$ is smooth and $f' : X' \to Y'$ is a Mori Fiber space with del Pezzo fibers meaning $\rho(X') - \rho(Y') = 1$. This will hold in the particular example when $X$ admits a global $1$-form with no zeros.
\par 
Since $R^2 f'_* \ul{\CC}$ is trivial this implies that the smooth fibers have $\rank{\NS(X'_y)} = 1$ and hence are isomorphic to $\P^2$. Suppose $y \in Y$ is a singular point of $f'$ we assume that there are $\omega_1, \dots, \omega_{m-1}$ with $m = \dim{Y}$ local $1$-forms around $y \in Y$ whose pullback to $X$ are pointwise linearly independent along the entire fiber. Therefore, slicing by a curve whose conormal bundle at $y$ is spanned by $\omega_1, \dots, \omega_{m-1}$ we obtain a slice $f : X'|_C \to C$ so that $X'|_C$ is smooth and the generic fiber is $\P^2$. Then we can apply the results of \cite{Fuj90} to conclude that the fibers are generically reduced (hence reduced since they are lci) \cite[Proposition 1.4]{Fuj90} and normal \cite[Theorem 3.1]{Fuj90}. Therefore, the singular locus is $0$-dimensional on each fiber. Let $Z \subset X'$ be the singular locus of $f' : X' \to Y$. Since $X' \to Y$ is a map of smooth varieties, by purity, Lemma~\ref{lem:purity}, each component $Z' \subset Z$ of the singular locus has $\codim_{X'} Z' \ge 1 + \dim{X} - \dim{Y} = 3$. Since $f'$ is flat and we showed that $Z$ intersects each fiber in a $0$-dimensional scheme we see that $Z \to Y$ is finite so its image has pure codimension $1$. Since $f'$ is smooth in codimension $1$ by assumption, we conclude that $f'$ is smooth.
\par 
Finally, if we moreover assume that the divisorial contractions $X_{Y'} \to X'$ are all blowdowns along \etale multisections over $Y$ then $X_{Y'} \to Y'$ is smooth if and only if $f'$ is smooth. Since $X_{Y'} \to Y'$ is the base change of $f$ along $Y' \to Y$ this shows that $f$ is smooth whenever $f'$ is. 

In summary, we have proved the following.

\begin{prop} \label{prop:del_pezzo_main}
Let $f : X \to Y$ be a del Pezzo fibration between smooth varieties. Assume the following hold:
\begin{enumerate}
    \item $f$ is smooth in codimension $1$
    \item for any finite \etale cover $Y' \to Y$ the base change $X_{Y'}$ admits a sequence of divisorial contractions each of which is the blowdown of an \etale multisection over $Y'$ such that we obtain a Mori Fiber space del Pezzo over $Y'$
    \item at each $y \in Y$ there are locally defined $1$-forms $\omega_1, \dots, \omega_{m-1}$ whose pullback to $X$ are pointwise linearly independent on a neighborhood of the fiber $X_y$
\end{enumerate}
Then $f$ is smooth.
\end{prop}

\section{Main results}

In this section, we first verify the $\dim{X} = 4$ case of a conjecture of Hao \cite[Conj.~1.5]{Hao23}.

\begin{thm}\label{thm:smooth_map_to_simpleAV}
Let $f : X \to A$ be a morphism from a smooth projective 4-fold $X$ to a simple abelian variety $A$. The following are equivalent:
\begin{enumerate}
\item[(a)] there exists a holomorphic one-form $\omega \in H^0(A, \Omega^1_A)$ such that $f^{\ast}\omega$ is nowhere vanishing;
\item[(b)] $f : X \to A$ is smooth.
\end{enumerate}
Moreover, when this holds then $X$ has the following structure:
\begin{enumerate}
    \item if $\dim{A} = 2$ then there exists a factorization $X \to X^{\min} \to A$ such that $X \to X^{\min}$ is a composition of blowups at \etale multi-sections of $X^{\min} \to A$ and either
    \begin{enumerate}
        \item $K_{X^{\min}}$ is nef in which case $X^{\min} \cong Z \times^G A'$ where $A' \to A$ is an isogeny, $G \subset \ker{(A' \to A)}$ and $Z$ is a smooth proper surface with a $G$-action
        \item $X^{\min}$ admits a conic bundle $X^{\min} \to Y$ and the Stein factorization $Y \to A' \to A$ is the composition of a smooth isotrivial fibration in curves and an isogeny
        \item the Stein factorization $X^{\min} \to A' \to A$ is the composition of a smooth del Pezzo fibration $X^{\min} \to A'$ and an isogeny $A' \to A$
    \end{enumerate}
    \item if $\dim{A} = 3$ then the Stein factorization $X \to A' \to A$ is the composition of a smooth isotrivial fibration in curves and an isogeny 
    \item if $\dim{A} \ge 4$ then $X$ is an abelian variety and $X \to A$ is an isogeny.
\end{enumerate}
\end{thm}

\begin{proof}
Note that the cases $\dim{A} \le 1$ are trivial so we may assume that $\dim{A} \ge 2$. 
We first run MMP on $X$. Let $g : X \to Y$ be an extremal contraction and consider the possibilities outlined in Theorem~\ref{thm:contractions}. Since $\dim{A} \ge 2$ we immediately rule out $(5)$. Likeiwse, in case $(2)$ since $Y$ has rational singularities we get a factorization
\begin{center}
\begin{tikzcd}
    X \arrow[rr, "f"] \arrow[rd, "g"]  & & A
    \\
    & Y \arrow[ru]
\end{tikzcd}
\end{center}
and the $1$-form $f^* \omega = g^* \omega_Y$ satisfies $\omega_Y|_E$ is nonzero so $E \subset Y$ maps non-constantly to $A$ since $E \to A$ pulls back $\omega$ to a nonzero form. This contradicts the simplicity of $A$. Therefore, $Y$ is smooth projective and carries a map $f_Y : Y \to A$ so that the $1$-form $f^*_Y \omega$ is nowhere vanishing. In case (1), $g$ is the blowup at a smooth surface $S \subset Y$ so by Lemma~\ref{lemma:blowup_nowhere_vanishing} $\omega_Y|_S$ is nowhere vanishing so $S \to A$ is \etale by Corollary~\ref{cor:finite_implies_etale} so, in particular, $\dim{A} = 2$. Hence if $\dim{A} > 2$ there are no contractions except those of fiber type. Either we run (1) finitely many times to get a minimal model $X \to X^{\min}$ or we reach a conic bundle (3) or a del Pezzo fibration (4). If $X^{\min}$ is minimal then we can apply \cite[Theorem 4.4]{Church24} (recall that Fujino has proven the abundance conjecture for irregular $4$-folds \cite{Fuj10}) to conclude that $X^{\min} \to A$ is smooth and $X^{\min}$ satisfies the decomposition required in (1a). 
\par 
Suppose $g : X^{\min} \to Y$ is a del Pezzo fibration. Then by Theorem~\ref{thm:contractions} (4) $g : X^{\min} \to Y$ is flat and $Y$ is smooth. Hence $\dim{Y} = \dim{A} = 2$. Since $h : Y \to A$ satisfies $h^* \omega = \omega_Y$ is nowhere vanishing we see that $h$ is \etale by Corollary~\ref{cor:finite_implies_etale} and hence $Y$ is a simple abelian surface. By the discussion in \S~\ref{section:delPezzo} we will conclude that $g$ is smooth. Indeed, to apply Proposition~\ref{prop:del_pezzo_main} we need to verify the three hypotheses hold in our situation: 
\begin{enumerate}
    \item $g$ is smooth is codimension $1$ by Lemma~\ref{prop:smooth_in_codim_1}.
    \item after performing a base change along $A' \to A$ to trivialize $R^2 g_* \ul{\CC}$ we can repeat the above argument to show that the only contractions possible are blowdowns at \etale multi-sections and hence we again end up at a smooth $X'^{\min}$ and a map $X'^{\min} \to Y$ which is either a conic bundle (in which case we have reduced to case (3) treated below) or is a del Pezzo fibration with generic fiber $\P^2$ in which case the second assumption is verified.

    \item Since there is a global nonvanishing $1$-form $g^* \omega_Y$ this implies that at each point $y \in Y$ its restriction proves a local $1$-form on $Y$ whose pullback to $X$ is nonvanishing in a neighborhood of the fiber, satisfying exactly the conditions needed for a smooth slice down to $Y$ being a curve to exist. 
\end{enumerate}
\par 
Suppose $g : X^{\min} \to Y$ is a conic bundle, i.e.\ arises from (3) so $Y$ is a smooth projective $3$-fold and $Y \to A$ pulls back $\omega$ to a nowhere vanishing $1$-form. Therefore by \cite[Theorem 1.3]{HS21(1)} we must have $Y \to A$ is either \etale or a smooth isotrivial curve fibration (there are no blowdowns because each has center at an elliptic curve not contracted to $A$ but $A$ is simple and in the case of a conic bundle the discriminant locus is a union of translates of elliptic curves in $A$ of which there are none) followed by an isogeny. By Proposition~\ref{prop:discriminat_locus} $\Delta_g \subset Y$ is a reduced divisor and $[\Delta_g] = -c_1(\E)$ satisfies the relations of Lemma~\ref{lemma:conic_bundle_formulas} because $X$ carries a nonvanishing $1$-form and hence $\chi(X, \Omega_X^p) = 0$ for all $p$ and so does $Y$ so $c_3(Y) = 0$ and furthermore since $Y \to A$ is smooth we conclude that $c_2(Y) = 0$. 
\par 
Now we consider the structure of $Y$. Since $Y \to A$ is smooth, we can apply Stein factorization to get $Y \to A' \to A$ with $A' \to A$ \etale and hence an isogeny. Replacing $A$ by $A'$ we can assume that $Y \to A$ has connected fibers. Let $X'$ be the base change of $X^{\min}$ along $A' \to A$.
Let us consider the cases individually:
\begin{enumerate}
    \item Suppose $\dim{A} = 3$ then $Y \to A$ is \etale so it is an isomorphsim. Since $\Delta_g \subset Y$ is an effective divisor, it is ample if and only if it is nonempty. However, $c_1(Y)$ so the formulas of Lemma~\ref{lemma:conic_bundle_formulas} give $\Delta_g^3 = 0$ hence $\Delta_g = \varnothing$ and thus $g$ is a smooth conic bundle.
    \item Suppose $\dim{A} = 2$ and $\pi : Y \to A$ is a smooth conic bundle then there is a decomposition of the Neron-Severi group
    \[ \NS(Y)_{\RR} = \pi^* \NS(A)_{\RR} \oplus \RR K_{Y/A} \] 
    and $K_{Y/A}$ restricts non-trivially to the fibers. Therefore, we can decompose numerically
    \[ [\Delta_g] = \pi^* D + n K_{Y/A} \] 
    First, note that
    \[ [\Delta_g]^3 = (\pi^* D^3) + 3 (\pi^* D^2) \cdot n K_{Y/A} + 3 \pi^* D \cdot (n K_{Y/A})^2 + n^3 K_{Y/A}^3  = n^3 K_{Y/A}^3 - 6 D^2 n \]
    because $D^3 = 0$ for dimension reasons, and $\pi^* D \cdot K_{Y/A}^2 = D \cdot \pi_* K_{Y/A}^2 = 0$ because $\pi$ is a smooth conic bundle and $-\pi_* K_{Y/A}^2$ is the class of its discriminant locus.
    Plugging into the first equation of Lemma~\ref{lemma:conic_bundle_formulas} and recalling that $[\Delta_g] = -c_1(\E)$ we obtain 
    \[ [\Delta_g] \cdot K_{Y/A}^2 = [\Delta_g]^3 = n^3 K_{Y/A}^3 - 6 D^2 n \]
    and $[\Delta_g] \cdot K_{Y/A}^2 = n K_{Y/A}^3$. Therefore,
    \[ 6 n D^2 = (n - n^3) K_{Y/A}^3 \]
    Plugging into the second formula, $K_{Y/A} [\Delta_g]^2 = K_{Y/A}^2 [\Delta_g]$, we obtain
    \[ n^2 K_{Y/A}^3 - 2 D^2 = n K_{Y/A}^3 \]
    using that $K_{Y/A} \cdot \pi^* D^2 = - 2 D^2$ and $K_{Y/A}^2 \cdot \pi^* D = 0$. These equations imply that $D^2 = 0$. Since $\Delta_g$ is effective and $A$ is simple we conclude that $X \to A$ is smooth.
    \item if $\dim{A} = 2$ and $Y \to A$ is a smooth isotrivial curve fibration of genus $\ge 1$ then it is finite \etale locally trivial so there is an isogeny $A' \to A$ so that $Y_{A'} \cong C \times A'$. Consider the base change $g' : X_{A'} \to C \times A'$ conic bundle. By applying generic smoothness to $X \to C$, we see that for a generic choice of $p \in C$ the divisor $\Delta_{g'|_{X_{A',p}}} := \Delta_{g'} \cap (\{ p \} \times A')$ is the discriminant locus of a conic bundle between smooth varieties. Since $g' : X_{A', p} \to A'$ carries the nonvanishing $1$-form $g'^* \omega$, by the analogue of Lemma~\ref{lemma:conic_bundle_formulas} for conic bundle threefolds, $[\Delta_{g'|_{X_{A',p}}}]^2 = 0$. However, $A$ is a simple abelian variety so no nonzero effective divisor $D \subset A$ satisfies $D^2 = 0$. Therefore, $\Delta_{g'} \cap (\{ p \} \cap A')$ is empty for generic $p \in C$ hence the image of $\Delta_{g'} \to C$ must be a finite collection of points. Since $\Delta_{g'}$ is a divisor, the projection $\Delta_{g'} \to A'$ is a disjoint union of isomorphisms. For each component $D \subset \Delta_{g'}$ the fibers of $X' \to Y$ are a pair of lines meeting at $1$ point since double fibers only appear only at singular points of $\Delta_f$ by \cite[Theorem 4.4]{Tan23}). Hence, the moduli space of lines over $D$ is an \etale double cover of $A'$. Hence passing to a further isogeny, we may assume that the monodromy over each $D$ is trivial. The preimage in $X'$ over each $D \subset \Delta_g$ is the union $E = E_1 \cup E_2$ of two divisors ruled over $D$. Since $D$ is a movable divisor, any line $\ell \subset E_1$ in the ruling satisfies $\ell \cdot E = 0$ but $\ell \cdot E_2 = 1$ so $\ell \cdot E_1 = -1$. The opposite is true for lines in $E_2$ so these curves are numerically distinct and $K_{X'}$-negative. Therefore, there is a divisorial contraction presenting $X'$ as the blowup of $X''$ along a section over $D$. See the argument of \cite{HS21(1)} on p.26 for details in 3-fold setting. Repeating this, we reduce to the case that $X \to Y = C \times A'$ is a smooth conic bundle.
\end{enumerate}

In total, $X \to A$ remains smooth because, possibly after performing a base change $X_{A'} \to A'$ it factors as $X_{A'} \to X' \to A'$ where $X' \to A'$ is a smooth isotrivial bundle and $X_{A'} \to X'$ is a blowdown at \etale multisections $S \subset X'$ of $X' \to A'$. These blowups preserve smoothness because every $1$-form $\omega \in H^0(A, \Omega_A)$ pulls back to a nonvanishing form under the \etale map $S \to A'$.
\end{proof}

\begin{thm}\label{thm:structure_no_elliptic_factor}
Let $X$ be a smooth projective $4$-fold with a nowhere vanishing holomorphic $1$-form $\omega$. Assume $\Alb_X$ does not contain an elliptic curve factor. Then there exists a sequence of blow-downs $\pi : X \to X'$ whose centers are smooth surfaces along which $\omega$ is nonvanishing and a smooth map $X' \to A$ to an abelian variety.
\end{thm}

\begin{proof}
Run MMP on $X$ to get an $g : X \to Y$ extremal contraction which must be one of the possibilities outlined in Theorem~\ref{thm:contractions}. Let $A$ be a maximal quotient of $\Alb_X$ so that the $1$-form $\omega$ is pulled back through $\Alb_X \to A$. Since $\dim{A} \ge 2$ we immediately rule out $(5)$. Likeiwse, in case $(2)$ since $Y$ has rational singularities we get a factorization
\begin{center}
\begin{tikzcd}
    X \arrow[rr, "f"] \arrow[rd, "g"]  & & A
    \\
    & Y \arrow[ru]
\end{tikzcd}
\end{center}
and the $1$-form $f^* \omega = g^* \omega_Y$ satisfies $\omega_Y|_E$ is nonzero so $E \subset Y$ maps non-constantly to $A$ since $E \to A$ pulls back $\omega$ to a nonzero form. Since we assumed $\Alb_X$ has no elliptic factor, this does not occur. Therefore, $Y$ is smooth projective and carries a map $f_Y : Y \to A$ so that the $1$-form $f^*_Y \omega$ is nowhere vanishing. In case (1), $g$ is the blowup along a smooth surface $S \subset Y$ so by Lemma~\ref{lemma:blowup_nowhere_vanishing} $\omega_Y|_S$ is nowhere vanishing. Either we run (1) finitely many times to get a minimal model $X \to X^{\min}$ or we reach a conic bundle (3) or a del Pezzo fibration (4). If $X^{\min}$ is minimal, then we can apply \cite[Theorem 4.4]{Church24} (again using Fujino's proof of the abundance conjecture for irregular $4$-folds \cite{Fuj10}) to conclude that $X^{\min} \to A \to B$ is smooth for some quotient $A \to B$.
\par 
As before, we reduce to showing that if $X = X^{\min}$ has the structure of a del Pezzo fibration or a conic bundle and admits a non-vanishing $1$-form, then it has the structure required in conclusion of the theorem. 
\begin{enumerate}
    \item Let $f : X \to S$ be a del Pezzo fibration. Then $\Alb_{X} = \Alb_{S}$ so if $\omega$ is a nonvanishing $1$-form on $X$, it is pulled back through $f$ and hence must be a nonvanishing $1$-form on $S$. Since $\Alb_X = \Alb_S$ has no abelian factor, $S \to \Alb_S$ must be generically finite onto its image. Indeed, otherwise it would factor through a curve of positive genus contradicting the existence of a $1$-form with no zero. Hence the image of $S \to \Alb_S$ is a surface and is fibered in tori by \cite[Theorem 1.4]{Schreieder21} hence is an abelian surface. Furtherore, this abelian surface is simple because $\Alb_S$ has no elliptic quotients. Therefore, applying Corollary~\ref{cor:finite_implies_etale} we conclude that $S$ is a simple abelian surface so $S \iso \Alb_S$. Therefore, we reduce to exactly the case of Theorem~\ref{thm:smooth_map_to_simpleAV} which proves that $f$ is smooth. 
    \item Let $f : X \to S$ be a conic bundle. As before, $\Alb_X = \Alb_S$ and $\omega = f^* \omega_S$ with $\omega_S$ a nonvanishing $1$-form. By \cite[Theorem 1.3]{HS21(1)} -- using that $\Alb_X = \Alb_S$ has no elliptic factors -- $S$ has the structure of a smooth map $g : S \to A$ to an abelian variety such that 
    \begin{enumerate}
        \item if $\dim{A} = 3$ then $g$ is an isogeny, 
        \item if $\dim{A} = 2$ then $g$ is a smooth isotrivial curve fibration.
    \end{enumerate}
    Note that, in either case, $A$ is simple because $\Alb_S$ has no elliptic quotient and in the first case $A$ is isogenous to $\Alb_S$ and an abelian $3$-fold with no $1$-dimensional factor cannot be decomposable. 
    In the first case, we are done by appealing to Theorem~\ref{thm:smooth_map_to_simpleAV}. In the second case, if $\omega$ were pulled back along $g$ we would be done by application of the same theorem. This works if $g$ is a $\P^1$-bundle. Otherwise, $g$ is strongly isotrivial in the sense that after an isogeny $A' \to A$ it is a product $C \times A'$ for $C$ a curve of genus $\ge 2$. Then $f : X_{A'} \to C \times A'$ is a conic bundle such that there exists a $1$-form $\omega$ without any zero. The smoothness of $f$ in this scenario is exactly what we proved in part (3) of the proof of Theorem~\ref{thm:smooth_map_to_simpleAV}. 
\end{enumerate}
\end{proof}

\bibliographystyle{alpha}
\bibliography{refs.bib}

\end{document}